\documentclass[11pt]{article}

\usepackage{cite}
\usepackage{amsthm,amsmath,amssymb}
\usepackage{algorithmic}
\usepackage{graphicx}
\usepackage{algorithm,algorithmic}
\usepackage{hyperref}
\usepackage{textcomp}
\usepackage{diagram}
\usepackage[normalem]{ulem}
\usepackage{bm}
\usepackage{mathrsfs}
\usepackage[utf8]{inputenc}
\allowdisplaybreaks[1]



\hypersetup{hidelinks=true}

\newtheorem{note}{Remark}[section]
\newtheorem{assum}{Assumption}[section]
\newtheorem{defn}{Definition}[section]
\newtheorem{lem}{Lemma}[section]
\newtheorem{thm}{Theorem}[section]
\newtheorem{pro}{Proposition}[section]
\newtheorem{cor}{Corollary}[section]

\title{Asymptotic Efficiency Analysis of the Recursive Least-Squares Algorithm for ARX Systems Without Projection}

\author{Xingrui Liu$^{*,**}$, Jieming Ke$^{*,**}$, and Yanlong Zhao$^{*,**}$ \\
\small * State Key Laboratory of Mathematical Sciences, \\
\small Academy of Mathematics and Systems Science, \\
\small Chinese Academy of Sciences, Beijing 100190, China \\ 
\small ** School of Mathematics Sciences, \\
\small University of Chinese Academy of Sciences, Beijing 100149, China  \\
\small \texttt{liuxingrui@amss.ac.cn; kejieming@amss.ac.cn; ylzhao@amss.ac.cn}
}
\date{}

\begin{document}

\maketitle

\begin{abstract}
This paper investigates the optimality analysis of the recursive least-squares (RLS) algorithm for autoregressive systems with exogenous inputs (ARX systems).
A key challenge in analyzing is managing the potential unboundedness of the parameter estimates, which may diverge to infinity.
Previous approaches addressed this issue by assuming that both the true parameter and the RLS estimates remain confined within a known compact set, thereby ensuring uniform boundedness throughout the analysis.
In contrast, we propose a new analytical framework that eliminates the need for such a boundness assumption. 
Specifically, we establish a quantitative relationship between the bounded moment conditions of quasi-stationary input/output signals and the convergence rate of the tail probability of the RLS estimation error.
Based on this technique, we prove that when system inputs/outputs have bounded twentieth-order moments, the RLS algorithm achieves asymptotic normality and the covariance matrix of the RLS algorithm converges to the Cramér-Rao lower bound (CRLB), confirming its asymptotic efficiency.
These results demonstrate that the RLS algorithm is an asymptotically optimal identification algorithm for ARX systems, even without the projection operators to ensure that parameter estimates reside within a prior known compact set.
\end{abstract}

\section{Introduction}
The RLS algorithm is one of the most established methods in adaptive control for system identification \cite{lee2018adaptive}. 
Significant research has focused on its theoretical convergence properties, which have been extensively discussed in the system identification literature \cite{XIAO20082207}.
In particular, \cite{ljung1987theory} proposed an analytical framework that the RLS algorithm for ARX systems is asymptotically normal, and the covariance matrix of the RLS algorithm converges to the CRLB as the sample size increases. 
This fits the two common senses of asymptotic efficiency:

\romannumeral1) Asymptotically efficient in the distribution sense \cite{van2000asymptotic}:  
An estimate is asymptotically efficient if its asymptotic distribution follows a normal distribution, where the variance of this normal distribution equals the limit of the CRLB. 

\romannumeral2) Asymptotically efficient in the covariance sense \cite{rao1961asymptotic}:  
An estimate is considered asymptotically efficient if its covariance matrix converges to the CRLB. 

Since the CRLB serves as a benchmark for the effectiveness of estimation processes, the fact that the RLS algorithm is asymptotically efficient demonstrates that it can be regarded as an asymptotically optimal identification algorithm \cite{friedlander1984computation}.

It is worth noticing that a key assumption of the analytical framework proposed in \cite{ljung1987theory} is that the unknown parameters reside within a prior known compact set, with parameter estimates constrained to this set through the use of a projection operation.
This assumption ensures the uniform boundedness of parameter estimates, thereby allowing the parameter estimation error to be treated as bounded in convergence analysis.

While the projection operator aids theoretical analysis, it introduces two challenges.  
First, it increases computational complexity by \( O(n^2) \) at each recursion, where \( n \) represents the dimension of the system regression vector \cite{Han2024}. 
This added overhead can be particularly burdensome in resource-constrained environments, such as embedded systems and real-time adaptive control applications \cite{6489677}.  
Second, the projection operator depends on prior knowledge of parameter bounds, which may be difficult to determine \cite{9720133}.  
%Eliminating the projection operator is therefore crucial both practically and theoretically, as it reduces computational costs and broadens the applicability of the RLS algorithm.

Several studies have examined the convergence properties of the RLS algorithm without the projection operator, focusing on aspects such as tracking performance for time-varying parameters \cite{NAEIMISADIGH2020107482}, the conditions for global asymptotic stability \cite{BRUCE2021105005}, and stability, robustness, and excitation in scenarios where disturbances affect both the regressor and regressand variables \cite{BIN2022105144}. 
However, to date, no study has considered the asymptotic efficiency analysis of the RLS algorithm in the absence of the projection operator.

Motivated by these gaps, this paper investigates the asymptotic efficiency analysis of the RLS algorithm for ARX systems, eliminating the prior assumption that system parameters reside within a known compact set, as required in \cite{ljung1987theory}.
To replace the role of uniform boundedness imposed by this prior assumption, we transform the convergence rate of the tail probability of the RLS estimation error into studying the convergence rate for the tail probability of the inverse covariance matrix associated with regression vectors.
Furthermore, a combinatorial method is introduced that derives a quantitative relationship between this rate and the bounded moment orders of quasi-stationary system inputs/outputs.
These theoretical developments lead to three key findings, which extend the foundational results of the RLS algorithm in \cite{ljung1987theory}:

\romannumeral1) When system inputs/outputs have bounded eighth-order moments, RLS estimates exhibit asymptotic normality, where the variance of the asymptotic distribution matches the limit of the CRLB. 
This confirms that the RLS algorithm is asymptotically efficient for ARX systems in the distribution sense.

\romannumeral2) When system inputs/outputs have bounded twentieth-order moments, the covariance matrix of RLS estimates attains the CRLB.
This confirms that the RLS algorithm is asymptotically efficient for ARX systems in the covariance sense.

\romannumeral3) When system inputs/outputs have bounded $4\gamma$-order moments, where $\gamma$ is an arbitrary positive integer, the RLS algorithm for ARX systems achieves the $L^{\gamma/2}$ convergence with a convergence rate of $O(1/k^{\gamma/4})$.

The remainder of this paper is organized as follows. 
Section \ref{sec_a} formulates the ARX system identification problem and outlines the RLS algorithm.
Section \ref{sec_c} establishes the asymptotic efficiency of the RLS algorithm without projection.
Section \ref{sec_f} is the summary and prospect of this paper.

In this paper, $\mathbb{R}$, $\mathbb{R}^{n}$, $\mathbb{R}^{n \times n}$, and $\mathbb{Z}$ are the sets of real numbers, $n$-dimensional real vectors and $n$-order matrices, and intergers respectively. 
For a constant $x$, $\vert x \vert$ denotes its absolute value.
For a vector $a=[a_{1},a_{2},\ldots ,a_{n}] \in \mathbb{R}^{n}$, $a_{i}$ denotes its $i$-th element, $a^{T}$ denotes its transpose; $\Vert a \Vert_{2}$ denotes its Euclidean norm. %, i.e, $\Vert a \Vert_{2} = (\sum_{i=1}^{n} a_{i}^{2})^{1/2}$.
For a matrix $A = [a_{i,j}]$, $i = 1,2,\ldots,n$, $j=1,2,\ldots,n$, $A^{T}$ denotes its transpose; $tr(A)$ denotes its trace; $A_{i,j}$ denotes its element in the $i$-th row and the $j$-th column, i.e, $A_{i,j} = a_{i,j}$; $\lambda_{\max}(A)$ and $\lambda_{\min}(A)$ denote its maximum eigenvalue and minimum eigenvalue respectively; $\Vert A \Vert_{1}$ denotes its $1$-norm, i.e., $\Vert A \Vert_{1} = \max_{1\leq j \leq n}\sum_{i=1}^{n} \vert a_{i,j} \vert$; $\Vert A \Vert_{2}$ denotes its $2$-norm, i.e., $\Vert A \Vert_{2} = \sqrt{ \lambda_{\max}(A^{T}A)}$. 
For a set $\Omega$, $\Omega^{\mathrm{c}}$ denotes its complement.
$\mathbb{P}$ denotes the probability operator.
$\mathbb{E}$ denotes the expectation operator.
The function $I_{\{\cdot\}}$ denotes the indicator function, whose value is $1$ if its argument (a formula) is true and $0$ otherwise.
$\mathcal{N}(\mu,\delta^2)$ denotes the Gaussian distribution with mean $\mu$ and variance $\delta^2$.

\section{Problem Formulation} \label{sec_a}

Consider an ARX system described by
\begin{align}\label{model_a}
A(z)y_{k} =  B(z)u_{k} + d_{k},\quad k \geq 1,
\end{align}
where $k$ is the time index; $u_{k} \in \mathbb{R}$ is the system input; $d_{k} \in \mathbb{R}$ is the unmeasurable noise; $A(z)$ is the $m$-th order polynomial and $B(z)$ is the $n$-th order polynomial, both expressed in terms of unit backward shift operator: $zy_{k} = y_{k-1}$ as $A(z) \triangleq 1+ a_{1}z + \ldots + a_{m}z^{m}$ and $B(z) \triangleq b_{1}z + \ldots + b_{n}z^{n}$; $a_{1} \in \mathbb{R}, \ldots, a_{m} \in \mathbb{R}$ and $b_{0},b_{1} \in \mathbb{R}, \ldots, b_{n} \in \mathbb{R}$ are the unknown parameters;  $y_{k}$ is the the system output.
The system inputs and outputs are stipulated such that $y_k = 0$ for $k \leq 0$ and $u_k = 0$ for $k < 0$.

We aim  to estimate the unknown parameter \(\theta \triangleq [a_{1}, \ldots, a_{m}, b_{1}, \ldots, b_{n}]^{T}\) based on the system inputs $u_{k}$ and system outputs $y_{k}$.
%\begin{figure}[h] as depicted in Fig.\ref{fig}
%\noindent\includegraphics[width=0.48\textwidth]{2.png}
%\caption{Identification problem for ARX systems}
%\label{fig}
%\end{figure}
To address this difficulty, the RLS algorithm is recognized as one of the most classical algorithms in system identification \cite{ljung1987theory}, and it is presented as follows:

Beginning with initial values $\hat{\theta}_0 \in \mathbb{R}^{n}$ and a positive definite matrix $P_{0} \in \mathbb{R}^{n \times n}$, the RLS algorithm for estimating $\theta$ is recursively defined at any $k \geq 1$ as follows:
\begin{align}
\label{LS} & \hat{\theta}_{k} = \hat{\theta}_{k-1} + a_{k}P_{k-1}\phi_{k}\left(y_{k}-\phi_{k}^{T}\hat{\theta}_{k-1}\right), \\
& a_{k} = \left(1+\phi_{k}^{T}P_{k-1}\phi_{k}\right)^{-1}, \\
\label{tPPPP} & P_{k} = a_{k}P_{k-1}\phi_{k}\phi_{k}^{T}P_{k-1},
\end{align}
where $\phi_{k} \triangleq [-y_{k-1}, \ldots, -y_{k-m}, u_{k-1}, \ldots, u_{k-n}]^{T}$ is the system regression vector and $\hat{\theta}_{k}$ is the RLS estimate of the unknown parameter $\theta$.

This paper focuses on analyzing the convergence properties, particularly the asymptotic efficiency of the RLS algorithm for ARX systems.
To proceed with our analysis, we introduce several fundamental assumptions, which are primarily based on those outlined in \cite{ljung1987theory}.

\begin{defn}\label{def_a}
For a process $\{ t_{k} \}_{k=1}^{\infty}$, the operator $\bar{\mathbb{E}}$ denotes
\[
\bar{\mathbb{E}}\left[t_{k}\right] \triangleq \lim\limits_{k \rightarrow \infty} \frac{1}{k}\sum_{l=1}^{k}\mathbb{E}\left[t_{l}\right].
\]
\end{defn}

\begin{defn}\label{def_b}
A process $\{ t_{k} \}_{k=1}^{\infty}$ is said to be quasi-stationary if it is subject to 
\begin{align*}
& \mathbb{E}\left[t_{k}\right] = m_{k}, \quad \mathop{\sup}\limits_{k} \left\vert m_{k} \right\vert < C, \quad k \geq 1, \\
& \mathbb{E}\left[t_{i}t_{j}\right] = R_{i,j}^{t}, \quad \mathop{\sup}\limits_{i,j} \left\vert R_{i,j}^{t} \right\vert < C,  \quad i,j \geq 1, 
\end{align*}
and the covariance function
\[
R_{\tau}^{t} \triangleq \bar{\mathbb{E}}\left[t_{k}t_{k-\tau}\right],   \quad \tau \in \mathbb{Z},
\]
exists, where $C$ is a positive constant.
\end{defn}

\begin{assum}\label{ass_a}%(Unified signal)
The system signasl are generated by a bounded deterministic sequence $\{r_{k}\}_{k=0}^{\infty}$ and a stochastic sequence $\{e_{k}\}_{k=0}^{\infty}$ of independent random variables with zero mean values, variance $\delta_{e}^{2}$, and bounded moments of order $4\gamma$, where $\gamma$ is a positive integer, %i.e., $\mathbb{E}[e_{k}^{n}] < \infty, n\in \mathbb{N}$, 
such that for some filters $\{ f_{i}^{(j)}(k) \}_{i=1}^{\infty}$, $j = 1,2,3,4$, $k \geq 1$,
\begin{align*}
& y_{k} = \sum_{i=1}^{\infty}  f_{i}^{(1)}(k)r_{k-i}  + \sum_{i=0}^{\infty} f_{i}^{(2)}(k)e_{k-i}, \\ 
& u_{k} = \sum_{i=0}^{\infty}  f_{i}^{(3)}(k)r_{k-i}  + \sum_{i=0}^{\infty} f_{i}^{(4)}(k)e_{k-i},
\end{align*}
where

\romannumeral1) (Uniform stability) the family of filters $\{ f_{i}^{(j)}(k) \}_{i=1}^{\infty}$, $j = 1,2,3,4$, $k \geq 1$ is uniformly stable, i.e., there exists a filter $\{ f_{i} \}_{i=1}^{\infty}$ such that
\[
\mathop{\sup}\limits_{k,j} \left\vert f_{i}^{(j)}(k) \right\vert \leq \left\vert f_{i} \right\vert, \quad i \geq 1,
\]
where $\sum_{i=1}^{\infty} \vert f_{i} \vert < \infty $;

\romannumeral2) (Joint quasi-stationary)
the system output $\{y_{k}\}_{k=1}^{\infty}$ and the system input $\{u_{k}\}_{k=0}^{\infty}$ both are quasi-stationary and, in addition, the cross-covariance function
\[
R_{\tau}^{yu} \triangleq \bar{\mathbb{E}}\left[y_{k}u_{k-\tau}\right],  \quad \tau \in \mathbb{Z},
\]
exists;

\romannumeral3) (Persistent excitation) 
the system regression vector $\{\phi_{k}\}_{k=1}^{\infty}$ satisfies the condition that %$\bar{\mathbb{E}}[ \phi_{k}\phi_{k}^{T}]$
\begin{align*}
& \bar{\mathbb{E}}[ \phi_{k}\phi_{k}^{T}]   \\
= &
\begin{bmatrix}
R_{0}^{y} & R_{1}^{y} & \cdots  & R_{m-1}^{y} &\cdots  & -R_{n-1}^{yu}\\
R_{1}^{y} &R_{0}^{y} &\cdots  & R_{m-2}^{y} & \cdots  & -R_{n-2}^{yu}\\
\vdots & \vdots & &  \vdots &  &   \vdots \\
R_{m-1}^{y} & R_{m-2}^{y} & \cdots &   -R_{m-1}^{uy} & \cdots &  -R_{0}^{uy} \\
\vdots & \vdots & &  \vdots &   &   \vdots \\
-R_{n-1}^{yu} & -R_{n-2}^{yu} &\cdots  &-R_{0}^{uy} & \cdots & R_{0}^{u}
\end{bmatrix}
\end{align*}
 is a positive definite matrix.
\end{assum}

\begin{assum}
\label{ass_b}
The filter $A(q)$ has no poles on or outside the unit circle
\end{assum}

\begin{assum}%(Gaussian noise)
\label{ass_c}
The system noise $\{d_{k}\}_{k=1}^{\infty}$ is a sequence of independent Gaussian random variables with zero mean and variance $\delta_{d}^{2}$.
\end{assum}

\begin{note}
Assumptions \ref{ass_a}--\ref{ass_c} are primarily derived from the conditions stated in Theorem 9.1 of \cite{ljung1987theory}, which derives the asymptotic covariance matrix of the RLS algorithm.
On the one hand, we relax the prior assumption that the system parameters belong to a known compact set. 
On the other hand, the condition of bounded $8$th-order moments for the stochastic sequence $\{e_{k}\}_{k=0}^{\infty}$ in system input/output is extended to $4\gamma$ boundedness, where $\gamma$ is a positive integer.
\end{note}

This paper aims to establish the asymptotic efficiency of the RLS algorithm for ARX systems without projection operators that ensure that parameter estimates reside within a prior known compact set under different values of \(\gamma\).

\begin{note}
The second condition essentially imposes assumptions on the bounded moments of the system input/output, as demonstrated in the following proposition:

\begin{pro}\label{pro_b2b}
Under Assumptions \ref{ass_a}, the system inputs and outputs have bounded moments of order $4\gamma$, i.e.,
\begin{align}\label{phik1}
\sup_{k}\mathbb{E}\left[\left\vert y_{k} \right\vert^{4\gamma}\right] < \infty, \quad \sup_{k}\mathbb{E}\left[\left\vert u_{k} \right\vert^{4\gamma}\right] < \infty.
\end{align}
\end{pro}

\begin{proof}
We prove $\sup_{k}\mathbb{E}[\vert y_{k} \vert_{2}^{4\gamma}] < \infty$ first.
By the $C_{r}$-inequality \cite{P}, since $\{e_{k}\}_{k=1}^{\infty}$ is independent and $\mathbb{E}[e_{k}] \equiv 0$, one can get
\begin{align*}
\mathbb{E}\left[\left\vert y_{k}\right\vert^{4\gamma}\right] \leq & 2^{4\gamma-1} \left\vert \sum_{i=1}^{\infty} f_{i}^{(1)}(k)r_{k-i}\right\vert^{4\gamma}  + 2^{4\gamma-1} \mathbb{E}\left[ \left\vert\sum_{i=0}^{\infty} f_{i}^{(2)}(k)e_{k-i}\right\vert^{4\gamma} \right].
\end{align*}
Since $\{ f_{i}^{(1)}(k) \}_{i=1}^{\infty}$ and $\{ f_{i}^{(2)}(k) \}_{i=1}^{\infty}$ are uniformly stable, $\{r_{k}\}_{k=0}^{\infty}$ is bounded, and $\{e_{k}\}_{k=0}^{\infty}$ has bounded moments of order $4\gamma$, we have
\begin{align*}
 & \left\vert \sum_{i=1}^{\infty} f_{i}^{(1)}(k)r_{k-i}\right\vert ^{4\gamma}
 \leq  \left( \sum_{i=1}^{\infty} \mathop{\sup}\limits_{k} \left\vert f_{i}^{(1)}(k) \right\vert \right) ^{4\gamma} \left(\mathop{\sup}\limits_{k} \left\vert r_{k}\right\vert\right)^{4\gamma}  < \infty,
\end{align*}
and
\begin{align*}
& \mathbb{E}\left[ \left\vert\sum_{i=0}^{\infty} f_{i}^{(2)}(k)e_{k-i}\right\vert^{4\gamma} \right] 
=  O\left( \left( \sum_{i=1}^{\infty} \mathop{\sup}\limits_{k} \left\vert f_{i}^{(2)}(k) \right\vert \right) ^{4\gamma} \cdot \sum_{i=1}^{4\gamma} \mathbb{E}\left[e_{k}^{i}\right] \right) < \infty.
\end{align*}
Hence, $\sup_{k}\mathbb{E}[\vert y_{k} \vert_{2}^{4\gamma}] < \infty$.
Similarly, it can be proven that  $\sup_{k}\mathbb{E}[\vert u_{k} \vert_{2}^{4\gamma}] < \infty$.
\end{proof}

The bounded moments of the system regression vector can be correspondingly derived:

\begin{cor}\label{cor_b2b}
Under Assumptions \ref{ass_a}, the system regression vector has bounded moments of order $4\gamma$, i.e.,
\begin{align}\label{phik1}
\sup_{k}\mathbb{E}\left[\Vert \phi_{k} \Vert_{2}^{4\gamma}\right] < \infty.
\end{align}
\end{cor}

\end{note}

\section{Main results}\label{sec_c}
This section develops an analytical framework to demonstrate the asymptotic efficiency of the RLS algorithm without projection operators. This analysis proceeds as follows:

\romannumeral1) We begin by deriving the convergence rate $O(1/k^{\gamma})$ for the tail probability of the inverse covariance matrix associated with regression vectors.

\romannumeral2) Based on the rate analysis, we establish the $L^{\gamma/2}$ convergence of the RLS algorithm for ARX systems by designing appropriate inequalities.

\romannumeral3)  Finally, the asymptotic efficiency of the RLS algorithm for ARX systems is established.

\subsection{Convergence rate for the tail probability of the inverse covariance matrix associated with regression vectors}

%which state that the stochastic sequence \(\{e_{k}\}_{k=0}^{\infty}\) in the system input \(\{u_{k}\}_{k=0}^{\infty}\) has bounded moments of order \(4\gamma\)
By applying a combinatorial mathematics method, we first present the convergence rate $O(1/k^{\gamma})$ for the tail probability of quasi-stationary system input/output sample means:
%We introduce a combinatorial mathematics method to establish the convergence rate $O(1/k^{\gamma})$ of quasi-stationary system input/output sample means:
\begin{thm}\label{lemma_a}
Under Assumptions \ref{ass_a}, for any integer $\tau$ and any positive constant $\varepsilon$, 
\begin{align}
\label{lemma_a_1}
& \mathbb{P}\left(\left\vert \frac{1}{k}\sum_{l=1}^{k} y_{l} y_{l-\tau} - R_{\tau}^{y} \right\vert > \varepsilon\right) = O\left(\frac{1}{k^{\gamma}}\right), \\
\label{lemma_a_2}
& \mathbb{P}\left(\left\vert \frac{1}{k}\sum_{l=1}^{k} u_{l} u_{l-\tau} - R_{\tau}^{u} \right\vert > \varepsilon\right) = O\left(\frac{1}{k^{\gamma}}\right), \\
\label{lemma_a_3}
& \mathbb{P}\left(\left\vert \frac{1}{k}\sum_{l=1}^{k} y_{l} u_{l-\tau} - R_{\tau}^{yu} \right\vert > \varepsilon\right) = O\left(\frac{1}{k^{\gamma}}\right).
\end{align}
%where $z_{k}$ is a unified expression of $y_{k}$ and $u_{k-1}$: $z_k \in \{y_{k},u_{k-1}\}$ for $k \geq 1$.
\end{thm}

\begin{proof}
Since the signals $\{y_{k}\}_{k=1}^{\infty}$ and $\{u_{k}\}_{k=0}^{\infty}$ are jointly quasi-stationary, $R_{\tau}^{y}$, $R_{\tau}^{yu}$ and $R_{\tau}^{u}$ exist directly.

%Without loss of any generality, %for any positive constant $\varepsilon$ and any integer $\tau$, 
%Without loss of any generality, 
We prove (\ref{lemma_a_1}) first. %where the following proof is primarily based on the methodology outlined in Appendix 2B (\!\cite{ljung1987theory}, p. 55).
Define $w_{k} \triangleq \sum_{i=1}^{\infty}  f_{i}^{(1)}(k)r_{k-i}$ and  $v_{k}  \triangleq \sum_{i=0}^{\infty} f_{i}^{(2)}(k)e_{k-i}.$
%\begin{align*}
%w_{k} \triangleq \sum_{i=1}^{\infty}  f_{i}^{(1)}(k)r_{k-i}, \quad  v_{k}  \triangleq \sum_{i=0}^{\infty} f_{i}^{(2)}(k)e_{k-i}.
%\end{align*} 
Then, from Assumptions \ref{ass_a}, $\{w_{k}\}_{k=1}^{\infty}$ is a bounded and deterministic, $\{v_{k}\}_{k=0}^{\infty}$ is  stochastic, and
\begin{align}\label{wwwvvv}
y_{k} = w_{k} + v_{k}.
\end{align}
Define $Y_{k} \triangleq \sum_{l=1}^{k} y_{l}y_{l-\tau} - \mathbb{E}[y_{l}y_{l-\tau}]$, $V_{k}  \triangleq \sum_{l=1}^{k} v_{l}v_{l-\tau} - \mathbb{E}[v_{l}v_{l-\tau}]$, $V_{k}^{(1)}  \triangleq \sum_{l=1}^{k} v_{l}w_{l-\tau}$ and $V_{k}^{(2)}  \triangleq \sum_{l=1}^{k} w_{l}v_{l-\tau}.$
%\begin{align*}
%& Y_{k} \triangleq \sum_{l=1}^{k} y_{l}y_{l-\tau} - \mathbb{E}\left[y_{l}y_{l-\tau}\right], \\
%& V_{k}  \triangleq \sum_{l=1}^{k} v_{l}v_{l-\tau} - \mathbb{E}\left[v_{l}v_{l-\tau}\right], \\ 
%& V_{k}^{(1)}  \triangleq \sum_{l=1}^{k} v_{l}w_{l-\tau}, \quad V_{k}^{(2)}  \triangleq \sum_{l=1}^{k} w_{l}v_{l-\tau}.
%\end{align*} 
For all $l \geq 1$, since $\mathbb{E}[v_{l}w_{l-\tau}] = 0$ and $\mathbb{E}[w_{l}v_{l-\tau}] = 0$, we have
\begin{align*}
\mathbb{E}\left[y_{l}y_{l-\tau}\right] & = \mathbb{E}\left[\left(w_{l} + v_{l}\right)\left(w_{l-\tau} + v_{l-\tau}\right)\right] \\
& = w_{l}w_{l-\tau} + \mathbb{E}\left[ v_{l}v_{l-\tau}\right],
\end{align*}
which implies
\begin{align}\label{pp11}
Y_{k} = V_{k} + V_{k}^{(1)} + V_{k}^{(2)}.
\end{align} 

%Besides, $\mathbb{P}(\vert 1/k\sum_{l=1}^{k}y_{l}y_{l-\tau} - R_{\tau}^{y} \vert > \varepsilon) = O(1/k^{3})$ equals $\mathbb{P}(\vert 1/k\sum_{l=1}^{k}Y_{k} \vert > \varepsilon) = O(1/k^{3})$.

%The following proof will be divided into two parts:

%$\mathbf{Part}$ $\mathbf{1:}$ Proof $\mathbb{E}[V_{k}^{6}] = O(k^{3})$, $\mathbb{E}[(V_{k}^{(1)})^{6}] = O(k^{3})$, and $\mathbb{E}[(V_{k}^{(2)})^{6}] = O(k^{3})$.

Now we prove $\mathbb{E}[V_{k}^{2\gamma}] = O(k^{\gamma})$.

Since $v_{k}  = \sum_{i=0}^{\infty} f_{i}^{(2)}(k)e_{k-i}$ and $V_{k} = \sum_{l=1}^{k} v_{l}v_{l-\tau} - \mathbb{E}[v_{l}v_{l-\tau}]$, $V_{k}$ can be expressed as
\begin{align}\label{V0}
V_{k} = \sum_{l=1}^{k}\sum_{i=0}^{\infty}\sum_{j=0}^{\infty}f_{i}^{(2)}(l)f_{j}^{(2)}(l-\tau)b(l,i,j,\tau).
\end{align}
where $b(l,i,j,\tau) = e_{l-i}e_{l-\tau-j}-\delta_{e}^{2}I_{\{i=\tau+j\}}$.
%Firstly, we proof $\mathbb{E}[Y_{k}^{6}] = O(k^{3})$ $\mathbb{E}[(V_{k}^{(1)})^{6}] = O(k^{3})$, and $\mathbb{E}[(V_{k}^{(2)})^{6}] = O(k^{3})$.
%By the H\"{o}lder's Inequality (\cite{P}, p. 158), we have 
%\begin{align*}
%& \mathop{\sup}\limits_{k} \mathbb{E}\left[e_{k}^{4}\right] \leq \mathop{\sup}\limits_{k} \left(\mathbb{E}\left[e_{k}^{12}\right]\right)^{1/3} = C_{e}^{1/3}, \\
%& \mathop{\sup}\limits_{k} \mathbb{E}\left[e_{k}^{6}\right] \leq \mathop{\sup}\limits_{k} \left(\mathbb{E}\left[e_{k}^{12}\right]\right)^{1/2} = C_{e}^{1/2}, \\
%& \mathop{\sup}\limits_{k} \mathbb{E}\left[e_{k}^{8}\right] \leq \mathop{\sup}\limits_{k} \left(\mathbb{E}\left[e_{k}^{12}\right]\right)^{2/3} = C_{e}^{2/3}.
%\end{align*} 
Since $\{e_{k}\}_{k=0}^{\infty}$ is a sequence of independent variables with zero mean values and bounded moments of order $4 \gamma$, we have
\begin{align}
\label{b0}  \mathbb{E}\left[b(l,i,j,\tau)\right] = 0,
\end{align}
and
\begin{align}
\label{b1} \mathop{\sup}\limits_{i,j,l,\tau} \mathbb{E}\left[\left(b(l,i,j,\tau)\right)^{2\gamma}\right] = O(1).
\end{align}

By (\ref{V0}), one can get
\begin{align*}
V_{k}^{2\gamma} = & \sum_{l_{1}=1}^{k}\cdots\sum_{l_{2\gamma}=1}^{k}\sum_{i_{1}=0}^{\infty}\cdots\sum_{i_{2\gamma}=0}^{\infty}\sum_{j_{1}=0}^{\infty}\cdots\sum_{j_{2\gamma}=0}^{\infty} \nonumber \\
& \prod_{t=1}^{2\gamma}f_{i_{t}}^{(2)}(l_{t})f_{j_{t}}^{(2)}(l_{t}-\tau)b(l_{t},i_{t},j_{t},\tau).
\end{align*}
Consequently,
\begin{align}\label{V2}
\mathbb{E}\left[V_{k}^{2\gamma}\right] \leq & \sum_{i_{1}=0}^{\infty}\cdots\sum_{i_{2\gamma}=0}^{\infty}\sum_{j_{1}=0}^{\infty}\cdots\sum_{j_{2\gamma}=0}^{\infty} \sum_{l_{1}=1}^{k}\cdots\sum_{l_{2\gamma}=1}^{k} \nonumber \\
& \left(\prod_{t=1}^{2\gamma} \left\vert f_{i_{t}}^{(2)}(l_{t}) \right\vert \left\vert f_{j_{t}}^{(2)}(l_{t}-\tau) \right\vert \right)
 \cdot \left\vert \mathbb{E}\left[ \prod_{t=1}^{2\gamma} b(l_{t},i_{t},j_{t},\tau) \right]\right\vert \nonumber \\
\leq & \sum_{i_{1}=0}^{\infty}\cdots\sum_{i_{2\gamma}=0}^{\infty}\sum_{j_{1}=0}^{\infty}\cdots\sum_{j_{2\gamma}=0}^{\infty}  \left(\prod_{t=1}^{2\gamma} \left(\mathop{\sup}\limits_{l} \left\vert f_{i_{t}}^{(2)}(l) \right\vert\right)  \left(\mathop{\sup}\limits_{l} \left\vert f_{j_{t}}^{(2)}(l) \right\vert\right) \right)\nonumber \\
& \sum_{l_{1}=1}^{k}\cdots\sum_{l_{2\gamma}=1}^{k}  \left\vert \mathbb{E}\left[ \prod_{t=1}^{2\gamma} b(l_{t},i_{t},j_{t},\tau) \right]\right\vert.
\end{align}
%Now we prove that $ \sum_{l_{1}=1}^{k} \cdots \sum_{l_{2\gamma}=1}^{k}  \vert \mathbb{E}[ \prod_{t_2=1}^{2\gamma} b(l_{t_2},i_{t_2},j_{t_2},\tau) ] \vert = O(k^{\gamma})$.
For each $i_{1}, i_{2}, \ldots, i_{2\gamma}, j_{1}, j_{2}, \ldots, j_{2\gamma} \in \{0,1,2,\ldots\}$ and $l_{1}, l_{2}, \ldots, l_{2\gamma} \in \{1,2,\ldots,k\}$, by (\ref{b1}) and H\"{o}lder's Inequality \cite{P}, there exists a positive constant $C_{b}$ such that
\begin{align*}
 & \left\vert \mathbb{E}\left[ \prod_{t=1}^{2\gamma} b(l_{t},i_{t},j_{t},\tau) \right]\right\vert 
\leq  \prod_{t=1}^{2\gamma} \left( \mathop{\sup}\limits_{i,j,l,\tau}\mathbb{E}\left[\left(b(l,i,j,\tau)\right)^{2\gamma}\right]\right)^{1/2\gamma} \leq C_{b}.
\end{align*}
For each given $i_{1}, i_{2}, \ldots, i_{2\gamma}, j_{1}, j_{2}, \ldots, j_{2\gamma}$, we define $n_{k}(i_{1},\ldots,i_{2\gamma},j_{1},\ldots,j_{2\gamma})$ as the total number of the non-zero terms in $\vert \mathbb{E}[\prod_{t=1}^{2\gamma}b(l_{t},i_{t},j_{t},\tau)] \vert$  for all possible index combinations \(l_{1}, l_{2}, \ldots, l_{2\gamma} \in \{1,2,\ldots,k\}\).
% the total number of instances where $\vert \mathbb{E}[\prod_{t=1}^{2\gamma}b(l_{t},i_{t},j_{t},\tau)] \vert$ is non-zero for all $l_{1} = 1,2,\ldots,k$, $l_{2} = 1,2,\ldots,k$, $\ldots$, $l_{2\gamma} = 1,2,\ldots,k$. 
Then, we have
\[\sum_{l_{1}=1}^{k}\cdots\sum_{l_{2\gamma}=1}^{k}  \left\vert \mathbb{E}\left[ \prod_{t=1}^{2\gamma} b(l_{t},i_{t},j_{t},\tau) \right]\right\vert \leq N_{k}C_{b},
\]
where
\[
N_{k} \triangleq \sup_{i_{1},\ldots,i_{2\gamma},j_{1},\ldots,j_{2\gamma}}n_{k}(i_{1},\ldots,i_{2\gamma},j_{1},\ldots,j_{2\gamma}).
\]
Besides, by $\sup_{l}\vert f_{i}^{(2)}(l) \vert\leq \vert f_{i} \vert$ and $\sum_{i=1}^{\infty} \vert f_{i} \vert < \infty $, there exists a positive constant $C_{f}$ such that 
\begin{align*}
\sum_{i=0}^{\infty}\mathop{\sup}\limits_{l} \left\vert f_{i}^{(2)}(l) \right\vert \leq C_{f}.
\end{align*}
Therefore, by (\ref{V2}), we have
\[
\mathbb{E}\left[V_{k}^{2\gamma}\right] \leq C_{f}^{4\gamma}C_{b}N_{k}.
\]
Thus, the problem of determining the upper bound order of the term $\mathbb{E}[V_{k}^{2\gamma}] $ with respect to \(k\) can be transformed into a combinatorial mathematical problem that calculates $N_{k}$.

By (\ref{b0}), for each $i_{1}, i_{2}, \ldots, i_{2\gamma}, j_{1}, j_{2}, \ldots, j_{2\gamma} \in \{0,1,2,\ldots\}$ and $l_{1}, l_{2}, \ldots, l_{2\gamma} \in \{1,2,\ldots,k\}$, it holds that $\vert \mathbb{E}[\prod_{t=1}^{2\gamma}b(l_{t},i_{t},j_{t},\tau)] \vert$ is non-zero only if each element in $\{b(l_{t}, i_{t},j_{t},\tau)\}_{t=1}^{2\gamma}$ is related to at least one other element within the same sequence. 
In this case, $\{b(l_{t}, i_{t},j_{t},\tau)\}_{t=1}^{2\gamma}$  can have at most \(\gamma\) pairwise independent elements.

Now, we use \(\gamma\) different cases to partition all non-zero instances of $\vert \mathbb{E}[\prod_{t=1}^{2\gamma}b(l_{t},i_{t},j_{t},\tau)] \vert$.
%we aim to identify the maximum number of pairs of independent elements within $\{b(l_{t}, i_{t},j_{t},\tau)\}_{t=1}^{2\gamma}$ in the case where $\vert \mathbb{E}[\prod_{t=1}^{2\gamma}b(l_{t},i_{t},j_{t},\tau)] \vert$ is non-zero.
Specifically, if there are no two independent elements in $\{b(l_{t}, i_{t},j_{t},\tau)\}_{t=1}^{2\gamma}$, we denote this situation as Case 1; if $\{b(l_{t}, i_{t},j_{t},\tau)\}_{t=1}^{2\gamma}$ contains at most two pairwise independent elements, we denote this situation as Case 2; $\ldots$
Since $\{b(l_{t}, i_{t},j_{t},\tau)\}_{t=1}^{2\gamma}$ contains at most \(\gamma\) pairwise independent elements, we can use \(\gamma\) different cases to partition all non-zero instances of $\vert \mathbb{E}[\prod_{t=1}^{2\gamma}b(l_{t},i_{t},j_{t},\tau)] \vert$.

Next, we will compute all possibilities in each Case \(\rho\), where \(\rho \in \{1, 2, \ldots, \gamma\}\). 
To begin with, we partition the sequence $\{b(l_{t}, i_{t},j_{t},\tau)\}_{t=1}^{2\gamma}$ into \(\rho\) distinct non-empty subsequences, the elements in each subsequence are pairwise related, and the elements drawn from different subsequences are independent of each other. 
The total number of possible partitions is represented by the Stirling number of the second kind \(S(2\gamma, \rho)\), which denotes the number of ways to partition \(2\gamma\) elements into \(\rho\) non-empty subsets \cite{graham1994concrete}.

It is worthy noticing that, if \( b(l_{\mu}, i_{\mu}, j_{\mu}, \tau) \) and \( b(l_{\nu}, i_{\nu}, j_{\nu}, \tau) \) are related, where \( \mu, \nu \in \{ 1, 2, \ldots, 2\gamma \} \) and \( \mu \neq \nu \), then one of the following conditions must hold:
\begin{align*}
&l_{\mu}-i_{\mu} = l_{\nu}-i_{\nu} \quad \mathrm{or} \quad l_{\mu}-i_{\mu} = l_{\nu}-\tau-j_{\nu} \quad \mathrm{or} \\
&l_{\mu}-\tau-j_{\mu} = l_{\nu}-i_{\nu} \quad \mathrm{or} \quad l_{\mu}-\tau-j_{\mu} = l_{\nu}-\tau-j_{\nu}.
\end{align*}
Thus, with specified values of \( i_{\mu}, j_{\mu}, i_{\nu}, j_{\nu} \), and \( l_{\mu} \), there are at most four possible values for \( l_{\nu} \) when \( b(l_{\mu}, i_{\mu}, j_{\mu}, \tau) \) and \( b(l_{\nu}, i_{\nu}, j_{\nu}, \tau) \) are related.

Since each element can take on at most \(k\) possible values, each partitioned subsequence can take on at most $4^{n_{sub}-1}k$ possible values, where $n_{sub} \leq 2\gamma$ represents the number of elements in this subsequence.
Therefore, the total number of possible value combinations for the \(\rho\) subsequences is at most \((4^{2\gamma-1} k)^{\rho}\).
By multiplying the number of partitions by the probability of each partition, we obtain the total probability for Case \(\rho\) is at most $S(2\gamma, \rho) (4^{2\gamma-1} k)^{\rho}.$

Since $S(2\gamma, \rho)$ is independent of $k$, we have
\[
N_{k} \leq \sum_{\rho=1}^{\gamma} S(2\gamma, \rho)\left(4^{2\gamma-1} k\right)^{\rho} = O(k^\gamma).
\]
which indicates that
\begin{align*}
\mathbb{E}\left[V_{k}^{2\gamma}\right] = O(k^{\gamma}).
\end{align*}
Similarly, it can be proven that $\mathbb{E}[(V_{k}^{(1)})^{2\gamma}] = O\left(k^{\gamma}\right)$ and $\mathbb{E}[(V_{k}^{(2)})^{2\gamma}] = O(k^{\gamma})$.
%, p. 157)
Then, by  (\ref{pp11}) and the $C_{r}$-inequality \cite{P}, we have
\begin{align*}
\mathbb{E}\left[Y_{k}^{2\gamma}\right] \leq & 2^{2\gamma-1} \mathbb{E}\left[V_{k}^{2\gamma}\right] +2^{2\gamma-1}\mathbb{E}\left[\left(V_{k}^{(1)}\right)^{2\gamma}\right]  + 2^{2\gamma-1}\mathbb{E}\left[\left(V_{k}^{(2)}\right)^{2\gamma}\right] \\
& = O\left(k^{\gamma}\right).
\end{align*}
In addition, %
by the Markov inequality \cite{P}, it holds that
\begin{align*}
\mathbb{P}\left( \left\vert Y_{k} \right\vert \geq k\varepsilon \right) \leq \frac{\mathbb{E}[Y_{k}^{2\gamma}] }{ \varepsilon^{2\gamma}k^{2\gamma}}  = O\left(\frac{1}{k^{\gamma}}\right).
\end{align*}
Since $\mathbb{P}( \vert\sum_{l=1}^{k}y_{l}y_{l-\tau} - R_{\tau}^{y}\vert/k > \varepsilon) =  \mathbb{P}( \vert Y_{k} \vert/k\geq\varepsilon)$, we have (\ref{lemma_a_1}) holds, where (\ref{lemma_a_2}) and (\ref{lemma_a_3}) can be proven analogously.
\end{proof}

Then, the inverse covariance matrix of regression vectors $P_{k}$ has the corresponding property:

\begin{cor}\label{lemma_b}
Under Assumptions \ref{ass_a}, for any positive constants $\varepsilon$,
\begin{align}
\label{lemma_b_1} & \mathbb{P}\left(\left\Vert  \bar{\mathbb{E}}\left[ \phi_{k}\phi_{k}^{T}\right] - \frac{P_{k}^{-1}}{k} \right\Vert_{2}> \varepsilon\right) = O\left(\frac{1}{k^{\gamma}}\right), \\
\label{lemma_b_2} & \mathbb{P}\left(\left\Vert P_{k} \right\Vert_{2} > \frac{2\Vert (\bar{\mathbb{E}}[ \phi_{k}\phi_{k}^{T}])^{-1} \Vert_{2}}{k}\right) = O\left(\frac{1}{k^{\gamma}}\right).
\end{align}
\end{cor}

\begin{proof}
By the Woodbury matrix identity \cite{ljung1987theory}, one can get
\begin{align}\label{PP1}
P_{k} = \left(\sum_{l=1}^{k} \phi_{l}\phi_{l}^{T} + P_{0}^{-1}\right)^{-1}.
\end{align} 
Hence, 
\[
\bar{\mathbb{E}}[ \phi_{k}\phi_{k}^{T}] - \frac{P_{k}^{-1}}{k} = \bar{\mathbb{E}}[ \phi_{k}\phi_{k}^{T}]  - \frac{1}{k}\sum_{l=1}^{k}\phi_{l}\phi_{l}^{T} - \frac{P_{0}^{-1}}{k}.
\]
%Note that
%\begin{align*}
%& \mathbb{P}\left(\left\Vert  \bar{\mathbb{E}}\left[ \phi_{k}\phi_{k}^{T}\right] - \frac{P_{k}^{-1}}{k} \right\Vert_{2}> \varepsilon\right) \\
%= & 1 - \mathbb{P}\left(\left\Vert   \bar{\mathbb{E}}\left[ \phi_{k}\phi_{k}^{T}\right] - \frac{P_{k}^{-1}}{k} \right\Vert_{2} \leq \varepsilon\right).
%\end{align*}
Since $P_{0}$ is a positive definitive matrix, when $k$ is sufficiently large, $\Vert P_{0}^{-1}/k \Vert_{2} \leq \varepsilon/2$.
%Define $\mathrm{I}_{k}(\varepsilon) \triangleq \{\Vert\bar{\mathbb{E}}[ \phi_{k}\phi_{k}^{T}]  -  \sum_{l=1}^{k}\phi_{l}\phi_{l}^{T}/k \Vert_{2} > \varepsilon/2\}$.
Then, for sufficiently large $k$, by DeMorgan's Law \cite{casella2024statistical} and Boole's Inequality \cite{casella2024statistical},  
we have
\begin{align*}
& \mathbb{P}\left(\left\Vert  \bar{\mathbb{E}}\left[ \phi_{k}\phi_{k}^{T}\right] - \frac{P_{k}^{-1}}{k} \right\Vert_{2}> \varepsilon\right) % \leq 1 - \mathbb{P}\left(\mathrm{I}_{k}^{\mathrm{c}}(\varepsilon)\right) \\
\\
\leq & \mathbb{P}\left(\left\Vert  \bar{\mathbb{E}}\left[ \phi_{k}\phi_{k}^{T}\right] - \frac{1}{k}\sum_{l=1}^{k}\phi_{l}\phi_{l}^{T}  \right\Vert_{2}> \frac{\varepsilon}{2}\right).
\end{align*}
By  Lemma \ref{lemma_aaa333} in Appendix A and Theorem \ref{lemma_a}, it holds that %$\Vert  \bar{\mathbb{E}}[ \phi_{k}\phi_{k}^{T}] - \sum_{l=1}^{k}\phi_{k}\phi_{k}^{T}/k \Vert_{2} = O((\Vert  \bar{\mathbb{E}}[ \phi_{k}\phi_{k}^{T}] - \sum_{l=1}^{k}\phi_{k}\phi_{k}^{T}/k \Vert_{1})^{1/2}(\Vert  \bar{\mathbb{E}}[ \phi_{k}\phi_{k}^{T}] - \sum_{l=1}^{k}\phi_{k}\phi_{k}^{T}/k \Vert_{\infty})^{1/2})= O(\max_{1 \leq i,j \leq n} \vert a_{i,j} \vert)$,
\begin{align*}
& \mathbb{P}\left(\left\Vert  \bar{\mathbb{E}}\left[ \phi_{k}\phi_{k}^{T}\right] - \frac{1}{k}\sum_{l=1}^{k}\phi_{l}\phi_{l}^{T} \right\Vert_{2}> \frac{\varepsilon}{2}\right) \\
%= & O\left( \mathbb{P}\left(\left\Vert  \bar{\mathbb{E}}\left[ \phi_{k}\phi_{k}^{T}\right]  - \frac{1}{k}\sum_{l=1}^{k}\phi_{k}\phi_{k}^{T} \right\Vert_{1} > \frac{\varepsilon}{2}\right)\right) \\
= & O\left(\mathbb{P}\left( \mathop{\max}\limits_{1 \leq i,j \leq n} \left\vert \left( \bar{\mathbb{E}}\left[ \phi_{k}\phi_{k}^{T}\right]  - \frac{1}{k}\sum_{l=1}^{k}\phi_{l}\phi_{l}^{T}\right)_{i,j} \right\vert> \frac{\varepsilon}{2}\right)\right) 
= O\left(\frac{1}{k^{\gamma}}\right),
\end{align*}
which implies that (\ref{lemma_b_1}) holds.

Furthermore, for $\Vert kP_{k} \Vert_{2}$, we obtain
\begin{align*}
\left\Vert kP_{k} \right\Vert_{2} = & \left\Vert kP_{k} - \left(\bar{\mathbb{E}}\left[ \phi_{k}\phi_{k}^{T}\right] \right)^{-1} +\left(\bar{\mathbb{E}}\left[ \phi_{k}\phi_{k}^{T}\right] \right)^{-1} \right\Vert_{2} \\
\leq & \left\Vert kP_{k} - \left(\bar{\mathbb{E}}\left[ \phi_{k}\phi_{k}^{T}\right] \right)^{-1} \right\Vert_{2} + \left\Vert \left(\bar{\mathbb{E}}\left[ \phi_{k}\phi_{k}^{T}\right] \right)^{-1} \right\Vert_{2} \\
= &  \left\Vert kP_{k} \left( \bar{\mathbb{E}}\left[ \phi_{k}\phi_{k}^{T}\right] - \frac{P_{k}^{-1}}{k} \right) \left(\bar{\mathbb{E}}\left[ \phi_{k}\phi_{k}^{T}\right] \right)^{-1} \right\Vert_{2} 
 + \left\Vert \left(\bar{\mathbb{E}}\left[ \phi_{k}\phi_{k}^{T}\right] \right)^{-1} \right\Vert_{2} \\
\leq & \left\Vert kP_{k} \right\Vert_{2} \left\Vert  \bar{\mathbb{E}}\left[ \phi_{k}\phi_{k}^{T}\right] - \frac{P_{k}^{-1}}{k} \right\Vert_{2} \left\Vert \left(\bar{\mathbb{E}}\left[ \phi_{k}\phi_{k}^{T}\right]\right)^{-1} \right\Vert_{2} 
 + \left\Vert \left(\bar{\mathbb{E}}\left[ \phi_{k}\phi_{k}^{T}\right] \right)^{-1} \right\Vert_{2}.
\end{align*}
Let $\varepsilon < 1/(2\Vert (\bar{\mathbb{E}}[ \phi_{k}\phi_{k}^{T}] )^{-1} \Vert_{2})$. Then, in the case where $\Vert  \bar{\mathbb{E}}\left[ \phi_{k}\phi_{k}^{T}\right]  - P_{k}^{-1}/k \Vert_{2} \leq \varepsilon$, we have
\begin{align*}
\left\Vert kP_{k} \right\Vert_{2} < \frac{1}{2}\left\Vert kP_{k} \right\Vert_{2} + \left\Vert \left(\bar{\mathbb{E}}\left[ \phi_{k}\phi_{k}^{T}\right]\right)^{-1} \right\Vert_{2},
\end{align*}
which follows that
\begin{align*}
\left\Vert P_{k} \right\Vert_{2} < \frac{2\Vert (\bar{\mathbb{E}}[ \phi_{k}\phi_{k}^{T}])^{-1} \Vert_{2}}{k}.
\end{align*}
Hence, we have (\ref{lemma_b_2}) holds.
\end{proof}

%This section systematically analyzes three core convergence properties of the RLS algorithm: almost sure convergence, asymptotic normality, and asymptotic efficiency. Together, these properties characterize the statistical learning performance of the algorithm. 
%The first property to be examined is almost sure convergence, which reflects the algorithm's capability to identify the parameters accurately.

%\begin{note}
%Based on Corollary \ref{lemma_b} indicates that, the, we will derive an explicit mapping between the bounded moment order of system inputs/outputs and the maximal order of the $L^{p}$ convergence of the RLS algorithm by designing appropriate inequality scaling, which replaces the role of uniform boundedness previously imposed by the prior assumption:
%\end{note}

\subsection{$L^{\gamma/2}$ convergence of the RLS algorithm for ARX systems}

Based on Corollary \ref{lemma_b}, the $L^{\gamma/2}$ convergence of the RLS algorithm for ARX systems is established as follows:

\begin{thm}\label{lemma_c3}
Under Assumptions \ref{ass_a}-\ref{ass_c}, the RLS estimate $\hat{\theta}_{k}$ converges to $\theta$ in the $L^{\gamma/2}$ sense with a convergence rate of $O(1/k^{\gamma/4})$, i.e.,
\begin{align}
\label{lemma_e_1} \mathbb{E}\left[ \left\Vert \tilde{\theta}_{k} \right\Vert_{2}^{\gamma/2} \right] = O\left(\frac{1}{k^{\gamma/4}}\right).
\end{align}
\end{thm}

\begin{proof}
%Next, we will proof $\mathbb{E}[\Vert\tilde{\theta}_{k}^{RLS}\Vert_{2}^{4\gamma}] = O(k^{2\gamma})$.
Define $L_{k} \triangleq \sum_{l=1}^{k}\phi_{l}d_{l} + P_{0}^{-1}\tilde{\theta}_{0}$.

We first prove that 
\begin{align}\label{Lkl}
\mathbb{E}\left[\left\Vert L_{k} \right\Vert_{2}^{2\gamma} \right] = O(k^{\gamma}).
\end{align}

By the $C_{r}$-inequality \cite{P}, we have 
\[
\left\Vert L_{k} \right\Vert_{2}^{2\gamma}  \leq 2^{2\gamma-1}\left( \left\Vert\sum_{l=1}^{k}\phi_{l}d_{l}  \right\Vert_{2}^{2\gamma} + \left\Vert P_{0}^{-1} \tilde{\theta}_{0} \right\Vert_{2}^{2\gamma}\right).
\]
Since the term $\Vert P_{0}^{-1} \tilde{\theta}_{0} \Vert_{2}^{2\gamma}$ is independent of $k$, we have
\[
\mathbb{E}\left[\left\Vert L_{k} \right\Vert_{2}^{2\gamma} \right] = O\left(\mathbb{E}\left[\left\Vert\sum_{l=1}^{k}\phi_{l}d_{l}  \right\Vert_{2}^{2\gamma}\right]\right).
\]
Now we examine the term $\mathbb{E}[\Vert\sum_{l=1}^{k}\phi_{l}d_{l}\Vert_{2}^{2\gamma}]$. Note that
%The term $\Vert\sum_{l=1}^{k}\phi_{l}d_{l}\Vert_{2}^{2\gamma}$ can be expressed by
\begin{align*}
\left\Vert\sum_{l=1}^{k}\phi_{l}d_{l}  \right\Vert_{2}^{2\gamma} = & \sum_{l_{1}=1}^{k}\cdots\sum_{l_{2\gamma}=1}^{k} \prod_{t_{1}=1}^{\gamma} \phi_{l_{2t_{1}-1}}^{T}\phi_{l_{2t_{1}}} \prod_{t_{2}=1}^{2\gamma} d_{l_{t_{2}}}.
\end{align*}
From Assumption \ref{ass_a}, for each index combination $l_{1}, l_{2}, \ldots, l_{2\gamma} \in \{1,2,\ldots,k\}$, by Corollary \ref{cor_b2b} and H\"{o}lder inequality \cite{P}, it holds that
\begin{align*}
&\mathbb{E}\left[ \prod_{t_{1}=1}^{\gamma} \phi_{l_{2t_{1}-1}}^{T}\phi_{l_{2t_{1}}} \prod_{t_{2}=1}^{2\gamma} d_{l_{t_{2}}} \right] \\
= & \mathbb{E}\left[ \prod_{t_{1}=1}^{\gamma} \phi_{l_{2t_{1}-1}}^{T}\phi_{l_{2t_{1}}} \right] \mathbb{E}\left[ \prod_{t_{2}=1}^{2\gamma} d_{l_{t_{2}}} \right] \\
\leq &  \left( \prod_{t_{1}=1}^{2\gamma} \left(\mathbb{E}\left[ \left\Vert \phi_{l_{t_{1}}} \right\Vert_{2}^{2\gamma}\right] \right)^{1/2\gamma}\right) \left( \prod_{t_{2}=1}^{2\gamma} \left(\mathbb{E}\left[ \left\vert d_{l_{t_{2}}} \right\vert^{2\gamma}\right] \right)^{1/2\gamma}\right)  = O(1).
\end{align*}
Therefore, the problem of determining the upper bound order of the term $\mathbb{E}[\Vert\sum_{l=1}^{k}\phi_{l}d_{l}\Vert_{2}^{2\gamma}]$ with respect to \(k\) can be transformed into a combinatorial mathematical problem which calculates the number of non-zero terms in $\mathbb{E}[\prod_{t_{1}=1}^{\gamma} \phi_{l_{2t_{1}-1}}^{T}\phi_{l_{2t_{1}}} \prod_{t_{2}=1}^{2\gamma} d_{l_{t_{2}}}]$  for all possible index combinations \(l_{1}, l_{2}, \ldots, l_{2\gamma} \in \{1,2,\ldots,k\}\).

Similarly with the approach of Theorem \ref{lemma_a}, one can get
\begin{align*}
\mathbb{E}\left[\left\Vert L_{k} \right\Vert_{2}^{2\gamma} \right] = O(k^{\gamma}).
\end{align*}
Then, by the Lyapunov inequality \cite{poznyak2009advanced}, it holds that
\begin{align}\label{uuyy661}
&\mathbb{E}\left[\left\Vert L_{k} \right\Vert_{2}^{\gamma}  \right] \leq \sqrt{\mathbb{E}\left[\left\Vert L_{k}\right\Vert_{2}^{2\gamma}  \right]} = O(k^{\gamma/2}).
\end{align}
By Lemma \ref{lemma_c2233} in Appendix A, one can get
\begin{align}\label{uuyy66}
\tilde{\theta}_{k} = P_{k}L_{k}.
\end{align}
Since $P_{0}$ is a positive definite matrix, by (\ref{PP1}), we have
\begin{align}\label{uuyy662}
\mathbb{E}\left[\left\Vert \tilde{\theta}_{k} \right\Vert_{2}^{2\gamma}  \right] \leq \left\Vert P_{0} \right\Vert_{2}^{2\gamma}\mathbb{E}\left[ \left\Vert L_{k} \right\Vert_{2}^{2\gamma} \right] = O(k^{\gamma}).
\end{align}
Define $\mathrm{\Omega}_{k} \triangleq \{\Vert P_{k} \Vert_{2} > 2\Vert (\bar{\mathbb{E}}[ \phi_{k}\phi_{k}^{T}])^{-1} \Vert_{2}/k\}$.
By Corollary \ref{lemma_b}, one can get
\begin{align}\label{uuyy663}
\mathbb{E}\left[ I_{\{\mathrm{\Omega}_{k}\}}^{2} \right] = \mathbb{P}(\mathrm{\Omega}_{k}) = O\left(\frac{1}{k^{\gamma}}\right).
\end{align}
Then, by (\ref{uuyy661})-(\ref{uuyy663}) and H\"{o}lder inequality \cite{P}, it holds that
\begin{align}\label{ine1}
\mathbb{E}\left[\left\Vert \tilde{\theta}_{k} \right\Vert_{2}^{\gamma}  \right] = & \mathbb{E}\left[\left\Vert \tilde{\theta}_{k}\right\Vert_{2}^{\gamma} I_{\{\Omega_{k}\}} \right] + \mathbb{E}\left[\left\Vert \tilde{\theta}_{k} \right\Vert_{2}^{\gamma} I_{\{\Omega_{k}^{c}\}} \right] \nonumber \\
\leq &  \sqrt{\mathbb{E}\left[\left\Vert \tilde{\theta}_{k} \right\Vert_{2}^{2\gamma} \right]\mathbb{E}\left[ I_{\{\Omega_{k}\}}^{2} \right]}  
 + \mathbb{E}\left[ \frac{2^{\gamma}\Vert (\bar{\mathbb{E}}[ \phi_{k}\phi_{k}^{T}])^{-1} \Vert_{2}^{\gamma} \left\Vert L_{k} \right\Vert_{2}^{\gamma}}{k^{\gamma}} \right] \nonumber \\
= & O(1) + O\left(\frac{1}{k^{\gamma/2}}\right) = O(1).
\end{align}
Repeat the H\"{o}lder inequality \cite{P}, we have
\begin{align}\label{ine2}
\mathbb{E}\left[\left\Vert \tilde{\theta}_{k} \right\Vert_{2}^{\gamma/2}  \right] = & \mathbb{E}\left[\left\Vert \tilde{\theta}_{k}\right\Vert_{2}^{\gamma/2} I_{\{\Omega_{k}\}} \right] + \mathbb{E}\left[\left\Vert \tilde{\theta}_{k} \right\Vert_{2}^{\gamma/2} I_{\{\Omega_{k}^{c}\}} \right] \nonumber \\
\leq &  \sqrt{\mathbb{E}\left[\left\Vert \tilde{\theta}_{k} \right\Vert_{2}^{\gamma} \right]\mathbb{E}\left[ I_{\{\Omega_{k}\}}^{2} \right]} 
 + \mathbb{E}\left[ \frac{2^{\gamma/2}\Vert (\bar{\mathbb{E}}[ \phi_{k}\phi_{k}^{T}])^{-1} \Vert_{2}^{\gamma/2} \left\Vert L_{k} \right\Vert_{2}^{\gamma/2}}{k^{\gamma/2}} \right] \nonumber \\
=& O\left(\frac{1}{k^{\gamma/4}}\right).
\end{align}
\end{proof}

%O\left(\frac{1}{k^{\gamma/2}}\right) + O\left(\frac{1}{k^{\gamma/4}}\right) 

\begin{note}
By establishing a quantitative relationship between the bounded moment conditions of quasi-stationary input/output signals $\gamma$ and the convergence rate of the tail probability of the RLS estimation error $\mathbb{P}(\Omega_{k})$, we more precisely characterize the convergence rate of $\mathbb{E}[\Vert \tilde{\theta}_{k}\Vert_{2}^{\gamma} I_{\{\Omega_{k}\}}]$, thereby replacing the role of uniform boundedness imposed by the projection operator.
\end{note}

\subsection{Asymptotic efficiency of the RLS algorithm for ARX systems}

Finally, the asymptotic efficiency of the RLS algorithm for ARX systems is established as follows:

\begin{thm}\label{lemma_c2} 
Under Assumptions \ref{ass_a}-\ref{ass_c}, 

\romannumeral1) If $\gamma \geq 2$, then the RLS estimate $\hat{\theta}_{k}$ is an asymptotically efficient estimate of $\theta$ in the distribution sense, i.e.,
\begin{align}
\label{lemma_d_1}  \sqrt{k} \tilde{\theta}_{k} \xrightarrow{d}  \mathcal{N}\left(0,  \delta_{d}^{2}\left(\bar{\mathbb{E}}\left[ \phi_{k}\phi_{k}^{T}\right]\right)^{-1} \right),
\end{align}
where $\xrightarrow{d}$ denotes convergence in distribution; 
\begin{align}
\label{theorem_2aaa} \delta_{d}^{2}\left(\bar{\mathbb{E}}\left[ \phi_{k}\phi_{k}^{T}\right]\right)^{-1} = \lim\limits_{k \to \infty}k\sigma_{\mathrm{CR}}\left(k\right);
\end{align}
$\sigma_{\mathrm{CR}}(k) = \delta_{d}^{2}(\sum_{l=1}^{k}\mathbb{E}[\phi_{l}\phi_{l}^{T}])^{-1}$ is the CRLB.

\romannumeral2) If $\gamma > 4$, then the RLS estimate $\hat{\theta}_{k}$ is an asymptotically efficient estimate of $\theta$ in the covariance sense, i.e.,
\begin{align}
\label{theorem_3}\lim\limits_{k \rightarrow \infty} k\left(\mathbb{E}\left[\tilde{\theta}_{k}\tilde{\theta}_{k}^{T} \right]-\sigma_{\mathrm{CR}}\left(k\right) \right)= 0.
\end{align}
\end{thm}
\begin{proof}
Define $J_{k}(\vartheta) \triangleq \sum_{l=1}^{k}(y_{l}-\phi_{l}^{T}\vartheta)^{2}/(2k) + P_{0}^{-1}\Vert \hat{\theta}_{0} - \vartheta\Vert_{2}^{2}/(2k)$, where $\vartheta \in \mathbb{R}^{m+n}$. Then, we have
\begin{align*}
\frac{\partial J_{k}(\vartheta) }{ \partial \vartheta}  = \frac{1}{k}\left( - \sum_{l=1}^{k}\phi_{l}y_{l}+ \sum_{l=1}^{k}\phi_{l}\phi_{l}^{T}\vartheta - P_{0}^{-1}\hat{\theta}_{0} +  P_{0}^{-1}\vartheta\right), 
\end{align*}
and
\begin{align*}
\frac{\partial^{2} J_{k}(\vartheta) }{ \partial \vartheta^{2}} =  \frac{1}{k}\left(\sum_{l=1}^{k}\phi_{l}\phi_{l}^{T} +  P_{0}^{-1} \right) = \frac{P_{k}^{-1}}{k}> 0.
\end{align*}
By Lemma \ref{lemma_c2233} in Appendix A, $\hat{\theta}_{k}$ can be expressed as
\begin{align}\label{RLS}
& \hat{\theta}_{k} = P_{k}\sum_{l=1}^{k}\phi_{l}y_{l} + P_{k}P_{0}^{-1}\hat{\theta}_{0}^{RLS},
\end{align} 
which indicates that $\partial J_{k}(\vartheta)/ \partial \vartheta  |_{\vartheta = \hat{\theta}_{k}} = 0.$
%\begin{align*}
%\frac{\partial J_{k}(\vartheta) }{ \partial \vartheta} \Big |_{\vartheta = \hat{\theta}_{k}} = 0.
% \end{align*} 
Then, by the mean value theorem \cite{zorich2016mathematical}, we have
\begin{align*}
\frac{\partial J_{k}(\vartheta) }{ \partial \vartheta} \Big |_{\vartheta = \hat{\theta}_{k}} - \frac{\partial J_{k}(\vartheta) }{ \partial \vartheta} \Big|_{\vartheta = \theta} = \frac{\partial^{2} J_{k}(\vartheta) }{ \partial \vartheta^{2}}\Big|_{\vartheta = \xi_{k}}  \tilde{\theta}_{k},
\end{align*}
where $\xi_{k} = c\hat{\theta}_{k} + (1-c)\theta$, $0 < c < 1$. It follows that
\begin{align}\label{uuuiii333}
\sqrt{k} \tilde{\theta}_{k} = -kP_{k}\sqrt{k}\frac{\partial J_{k}(\vartheta) }{ \partial \vartheta} \Big|_{\vartheta = \theta}.
\end{align}
Define $S_{k} \triangleq \sum_{l=1}^{k}\phi_{l}d_{l}/\sqrt{k}$.
Since $y_{k}-\phi_{k}^{T}\theta = d_{k}$, we have
\begin{align}\label{uuuiii222}
\lim\limits_{k\rightarrow\infty}\left( \sqrt{k}\frac{\partial J_{k}(\vartheta) }{ \partial \vartheta}\Big|_{\vartheta = \theta} - S_{k}\right)= 0.
\end{align}
By Assumption \ref{ass_a}, $\phi_{k}$ can be expressed as
\begin{align*}
\phi_{k} = \sum_{i=1}^{\infty} f_{i}^{(5)}(k)r_{k-i}  + \sum_{i=0}^{\infty} f_{i}^{(6)}(k)e_{k-i},
\end{align*}
where $f_{i}^{(5)}(k) \in \mathbb{R}^{m+n}$ and $f_{i}^{(6)}(k) \in \mathbb{R}^{m+n}$ are the filters.
For a positive integer $M$, we define%$\phi_{k}(M) \triangleq \sum_{i=1}^{\infty} f_{i}^{(5)}(k)r_{k-i}  + \sum_{i=0}^{M} f_{i}^{(6)}(k)e_{k-i}$, $\tilde{\phi}_{k}(M) \triangleq \phi_{k} - \phi_{k}(M) = \sum_{i=M+1}^{\infty} f_{i}^{(6)}(k)e_{k-i}$, $Z_{k}(M) \triangleq \sum_{l=1}^{k}(\phi_{l}(M)d_{l})/\sqrt{k}$ and $X_{k}(M) \triangleq \sum_{l=1}^{k}(\tilde{\phi}_{l}(M)d_{l})/\sqrt{k}$.
\begin{align*}
&\phi_{k}(M) \triangleq \sum_{i=1}^{\infty} f_{i}^{(5)}(k)r_{k-i}  + \sum_{i=0}^{M} f_{i}^{(6)}(k)e_{k-i}, \\
&\tilde{\phi}_{k}(M) \triangleq \phi_{k} - \phi_{k}(M) = \sum_{i=M+1}^{\infty} f_{i}^{(6)}(k)e_{k-i}, \\
&Z_{k}(M) \triangleq \sum_{l=1}^{k}\frac{\phi_{l}(M)d_{l}}{\sqrt{k}}, \\
& X_{k}(M) \triangleq \sum_{l=1}^{k}\frac{\tilde{\phi}_{l}(M)d_{l}}{\sqrt{k}}.
\end{align*}
Then, $S_{k}$ can be expressed as
\begin{align}
\label{SSDD}S_{k} = Z_{k}(M) + X_{k}(M).
\end{align}
Next, we will apply Lemma \ref{lemma_2} in Appendix A to (\ref{SSDD}).
First, we show that $Z_{k}(M)$ is asymptotically normal according to Lemma \ref{lemma_1} in Appendix A.
By Corollary \ref{cor_b2b}, for $\gamma \geq 2$, one can get
\[
\mathbb{E}\left[\left\Vert\phi_{k}\right\Vert_{2}^{8}\right] = O(1).
\]
Then, by the Lyapunov inequality \cite{poznyak2009advanced}, it holds that
\begin{align*}
\mathbb{E}\left[ \left\Vert \frac{\phi_{l}(M)d_{l}}{\sqrt{k}} \right\Vert_{2}^{4} \right]
 \leq 
\frac{\sqrt{\mathbb{E}[\Vert \phi_{l}(M)\Vert_{2}^{8}] \mathbb{E}[ \Vert  d_{l}\Vert_{2}^{8}]}}{k^{2}} = O\left(\frac{1}{k^{2}}\right),
\end{align*}
and %$\mathbb{E}[ \Vert (\phi_{l}(M)d_{l})/\sqrt{k} \Vert_{2}^{2} ] \leq (\mathbb{E}[ \Vert (\phi_{l}(M)d_{l})/\sqrt{k} \Vert_{2}^{4} ] )^{1/2} = O(1/k)$.
\begin{align*}
\mathbb{E}\left[ \left\Vert \frac{\phi_{l}(M)d_{l}}{\sqrt{k}} \right\Vert_{2}^{2} \right]
 \leq \sqrt{\mathbb{E}\left[ \left\Vert \frac{\phi_{l}(M)d_{l}}{\sqrt{k}} \right\Vert_{2}^{4} \right]} = O\left(\frac{1}{k}\right).
\end{align*}
Then, by $\mathbb{E}[\phi_{l}(M)d_{l}/\sqrt{k}] = 0$ and Lemma \ref{lemma_1} in Appendix A, one can get
%which indicates $\sum_{l=1}^{k} \mathbb{E}[\Vert (\phi_{l}(M)d_{l})/\sqrt{k} \Vert_{2}^{4} ] = O(1/k)$ and $\sum_{l=1}^{k} \mathbb{E}[\Vert (\phi_{l}(M)d_{l})/\sqrt{k} \Vert_{2}^{2} ] = O(1)$.
%\begin{align*}
%&\sum_{l=1}^{k} \mathbb{E}\left[ \left\Vert \frac{\phi_{l}(M)d_{l}}{\sqrt{k}} \right\Vert_{2}^{4} \right] = O\left(\frac{1}{k}\right), \\ &\sum_{l=1}^{k} \mathbb{E}\left[ \left\Vert \frac{\phi_{l}(M)d_{l}}{\sqrt{k}} \right\Vert_{2}^{2} \right] = O(1).
%\end{align*}
%Note that $\mathbb{E}[\phi_{l}(M)d_{l}/\sqrt{k}] = 0$ and $Z_{k}(M) = \sum_{l=1}^{k}(\phi_{l}(M)d_{l})/\sqrt{k}$. By Lemma 9A.1 in \cite{ljung1987theory}, it holds that
\begin{align*}
Z_{k}(M) \xrightarrow{d}  \mathcal{N}(0,Q_{M}),
\end{align*}
where %$Q_{M} = \lim_{k\rightarrow\infty} \mathbb{E}[ Z_{k}(M) (Z_{k}(M))^{T}]=  \delta_{d}^{2}\lim_{k\rightarrow\infty}\sum_{l=1}^{k} \mathbb{E}[\phi_{l}(M)(\phi_{l}(M))^{T}]/k.$
\begin{align*}
Q_{M} & = \lim\limits_{k\rightarrow\infty} \mathbb{E}\left[ Z_{k}(M) (Z_{k}(M))^{T}\right]  =  \delta_{d}^{2}\lim\limits_{k\rightarrow\infty}\frac{1}{k}\sum_{l=1}^{k} \mathbb{E}\left[\phi_{l}(M)\left(\phi_{l}(M)\right)^{T}\right].
\end{align*}
Note that
\begin{align*}
\mathbb{E}\left[\Vert X_{k}(M)\Vert_{2}^{2}\right] = & \frac{\delta_{d}^{2}}{k}\mathbb{E}\left[\sum_{l=1}^{k}\Vert \tilde{\phi}_{l}(M)\Vert_{2}^{2}\right] =  \frac{\delta_{d}^{2}}{k}\sum_{l=1}^{k}\sum_{i=M+1}^{\infty}\delta_{e}^{2} \left\Vert f_{i}^{(6)}(l)\right\Vert_{2}^{2} \\
\leq & \frac{\delta_{d}^{2}\delta_{e}^{2}}{k}\sum_{l=1}^{k}\left(\sum_{i=M+1}^{\infty} \mathop{\sup}\limits_{l}\left\Vert f_{i}^{(6)}(l)\right\Vert_{2}\right)^{2}.
\end{align*}
By $\sum_{i=0}^{\infty} \sup_{k} \Vert f_{i}^{(6)}(k)\Vert_{2} = O(1)$, we have %$\lim_{M\rightarrow\infty} \sum_{i=M+1}^{\infty} \sup_{l}\Vert f_{i}^{(6)}(l)\Vert_{2} = 0.$
\begin{align*}
\lim\limits_{M\rightarrow\infty} \sum_{i=M+1}^{\infty} \mathop{\sup}\limits_{l}\left\Vert f_{i}^{(6)}(l)\right\Vert_{2} = 0.
\end{align*}
Hence,
\begin{align*}
\lim\limits_{M\rightarrow\infty}\mathbb{E}\left[\Vert X_{k}(M)\Vert_{2}^{2}\right] =0.
\end{align*}
Therefore, by Lemma \ref{lemma_2} in Appendix A, one can get
\begin{align*}
S_{k} \xrightarrow{d}  \mathcal{N}\left(0,Q\right),
\end{align*}
where 
\begin{align*}
Q & = \lim\limits_{M\rightarrow\infty} Q_{M} =  \delta_{d}^{2} \lim\limits_{M\rightarrow\infty}\lim\limits_{k\rightarrow\infty}\frac{1}{k}\sum_{l=1}^{k} \mathbb{E}\left[\phi_{l}(M)\left(\phi_{l}(M)\right)^{T}\right].
\end{align*}
Note that %$\lim_{M\rightarrow\infty} \mathbb{E}[\phi_{l}(M)(\phi_{l}(M))^{T}] =  \lim_{M\rightarrow\infty} \delta_{e}^{2} \sum_{i=1}^{M} f_{i}^{(6)}(l)(f_{i}^{(6)}(l))^{T} +  \sum_{i=1}^{\infty} f_{i}^{(5)}(l)(f_{i}^{(5)}(l))^{T}r_{k-i}^{2} = \mathbb{E}[\phi_{l}\phi_{l}^{T}].$
\begin{align*}
& \lim\limits_{M\rightarrow\infty} \mathbb{E}\left[\phi_{l}(M)\left(\phi_{l}(M)\right)^{T}\right] \\
= &  \lim\limits_{M\rightarrow\infty} \delta_{e}^{2} \sum_{i=1}^{M} f_{i}^{(6)}(l)\left(f_{i}^{(6)}(l)\right)^{T} +  \sum_{i=1}^{\infty} f_{i}^{(5)}(l)\left(f_{i}^{(5)}(l)\right)^{T}r_{k-i}^{2} 
=  \mathbb{E}\left[\phi_{l}\phi_{l}^{T}\right].
\end{align*}
Hence,
\begin{align*}
Q & = \delta_{d}^{2} \lim\limits_{k\rightarrow\infty}\frac{1}{k}\sum_{l=1}^{k} \lim\limits_{M\rightarrow\infty}\mathbb{E}\left[\phi_{l}(M)\left(\phi_{l}(M)\right)^{T}\right] \\
& = \delta_{d}^{2} \bar{\mathbb{E}}\left[\phi_{k}\phi_{k}^{T}\right].
\end{align*}
Thus, we have $S_{k} \xrightarrow{d}  \mathcal{N}(0,\delta_{d}^{2} \bar{\mathbb{E}}[\phi_{k}\phi_{k}^{T}])$.
%\begin{align*}
%S_{k} \xrightarrow{d}  \mathcal{N}\left(0,\delta_{d}^{2} \bar{\mathbb{E}}\left[\phi_{k}\phi_{k}^{T}\right]\right).
%\end{align*}
By (\ref{uuuiii222}), one can get
\begin{align}\label{pyte}
\sqrt{k}\frac{\partial J_{k}(\vartheta) }{ \partial \vartheta}\Big|_{\vartheta = \theta} \xrightarrow{d}  \mathcal{N}\left(0,\delta_{d}^{2} \bar{\mathbb{E}}\left[\phi_{k}\phi_{k}^{T}\right]\right).
\end{align}
By Corollary \ref{lemma_b} and Borel-Cantelli lemma \cite{ash2014real}, it holds that
\[
\left\Vert kP_{k} \right\Vert_{2} = O(1), \quad \text{a.s.}
\]
and
\[
\lim\limits_{k \rightarrow \infty} \left\Vert \bar{\mathbb{E}}\left[ \phi_{k}\phi_{k}^{T}\right] - \frac{P_{k}^{-1}}{k}\right\Vert
=  0, \quad  \mathrm{a.s.}
\]
Then, by
\begin{align*}
\left\Vert kP_{k} - \left(\bar{\mathbb{E}}\left[ \phi_{k}\phi_{k}^{T}\right] \right)^{-1} \right\Vert_{2} 
= &  \left\Vert kP_{k} \left( \bar{\mathbb{E}}\left[ \phi_{k}\phi_{k}^{T}\right] - \frac{P_{k}^{-1}}{k} \right) \left(\bar{\mathbb{E}}\left[ \phi_{k}\phi_{k}^{T}\right] \right)^{-1} \right\Vert_{2} \\
\leq & \left\Vert kP_{k} \right\Vert_{2} \left\Vert  \bar{\mathbb{E}}\left[ \phi_{k}\phi_{k}^{T}\right] - \frac{P_{k}^{-1}}{k} \right\Vert_{2} \left\Vert \left(\bar{\mathbb{E}}\left[ \phi_{k}\phi_{k}^{T}\right]\right)^{-1} \right\Vert_{2},
\end{align*}
we have
\begin{align}\label{P989822222}
\lim\limits_{k \rightarrow \infty} \left\Vert  \left( \bar{\mathbb{E}}\left[ \phi_{k}\phi_{k}^{T}\right]\right)^{-1} - kP_{k} \right\Vert_{2}= 0, \quad  \mathrm{a.s.}
\end{align}
By (\ref{uuuiii333}), (\ref{pyte}) and (\ref{P989822222}), it follows that
\begin{align*}
\sqrt{k} \tilde{\theta}_{k} \xrightarrow{d}  \mathcal{N}\left(0,\delta_{d}^{2}\left(\bar{\mathbb{E}}\left[ \phi_{k}\phi_{k}^{T}\right] \right)^{-1} \right).
\end{align*}
%which indicates that (\ref{lemma_d_1}) holds.
Then, by Lemma \ref{lemma_aaa111} in Appendix A, one can get
\begin{align*}
\lim\limits_{k \rightarrow \infty} k \sigma_{\mathrm{CR}}\left(k\right) = \delta_{d}^{2}\left(\bar{\mathbb{E}}\left[\phi_{k}\phi_{k}\right]\right)^{-1}.
\end{align*}
Besides, by Theorem \ref{lemma_c3}, for $\gamma \geq 5$, we have 
%Note that $\mathbb{E}[\Vert \sqrt{k}\tilde{\theta}_{k} \Vert_{2}^{4}]  < \infty.$
\[
\mathbb{E}\left[ \left\Vert \sqrt{k}\tilde{\theta}_{k} \right\Vert_{2}^{5/2}\right]  = O\left(1\right).
\]
Then, by  Lemma \ref{lemma_aaa222} in Appendix A, it holds that
\begin{align*}
\lim\limits_{k\rightarrow\infty} k \mathbb{E}\left[\tilde{\theta}_{k}\tilde{\theta}_{k}^{T} \right] = \delta_{d}^{2}\left(\bar{\mathbb{E}}\left[ \phi_{k}\phi_{k}^{T}\right] \right)^{-1}.
\end{align*}
\end{proof}
\begin{note}
Since the CRLB represents a benchmark for the effectiveness of estimation processes \cite{friedlander1984computation}, Theorem \ref{lemma_c2} indicates that the RLS algorithm is asymptotically optimal for ARX systems even without projection to ensure that parameter estimates reside within a prior known compact set.
\end{note}

\section{Conclusion}\label{sec_f}
This paper provides an asymptotic efficiency analysis of the RLS algorithm for ARX systems.
By establishing a novel analytical framework, we relax the prior assumption that the system parameters belong to a known compact set, thereby eliminating the need for projection operators.
%It introduces an explicit mapping relationship that connects the order of the bounded moments of the quasi-stationary system input/output to the order of the moments of the convergent parameter estimation error. 

Here, we give some topics for future research. 
First, how to further relax the excitation conditions of the system input/output.
Additionally, exploring whether the RLS algorithm can achieve asymptotic efficiency in more general system models represents an important open question. 
And whether this framework can be extended to other identification algorithms for asymptotic efficiency analysis.

\appendix

\section*{Useful tools} \label{aap_a}

\setcounter{equation}{0}
\renewcommand{\theequation}{A.\arabic{equation}}
\setcounter{lem}{0}
\renewcommand{\thelem}{A.\arabic{lem}}

\begin{lem}\label{lemma_aaa222} (Theorem 4.5.2 in \cite{chung2000course}) If the random process $\{X_{n}\}$ converges to $X$ in distribution, and for some $p > 0$, $\sup_{n}\mathbb{E}[\Vert X_{n} \Vert_{2}^{p}] < \infty$, then, for each $r < p$,
\[
\lim\limits_{k \to \infty} \mathbb{E}[\Vert X_{n} \Vert_{2}^{r}] = \mathbb{E}[\Vert X \Vert_{2}^{r}] < \infty.
\]
\end{lem}

\begin{lem}\label{lemma_aaa333} (Theorem 5.6.18 in \cite{horn2012matrix}) For any matrix $A$,
\[
\max_{A \neq 0}\frac{\Vert A \Vert_{2}}{\Vert A \Vert_{1}} = \max_{A \neq 0}\frac{\Vert A \Vert_{1}}{\Vert A \Vert_{2}} < \infty.
\]
\end{lem}

\begin{lem}\label{lemma_1} (Lemma 9.A.1 in \cite{ljung1987theory})
Consider the sum of doubly indexed random variables $x_{N}(k)$: $Z_{N} = \sum_{t=1}^{N} x_{N}(t),$
where $\mathbb{E}[x_{N}(t)] = 0.$
%\[
%\mathbb{E}\left[x_{N}(t)\right] = 0.
%\]
Suppose that $\{ x_{N}(1),\ldots,x_{N}(s) \}$ and $\{ x_{N}(t),x_{N}(t+1),\ldots,x_{N}(n) \}$ are independent if $t-s > M$, where $M$ is an integer, that
\[
\mathop{\lim\sup}\limits_{N \to \infty} \sum_{k=1}^{N}\mathbb{E}\left[\Vert x_{N}(k) \Vert_{2}^{2}\right] < \infty
\]
and that
\[
\mathop{\lim\sup}\limits_{N \to \infty} \sum_{k=1}^{N}\mathbb{E}\left[\Vert x_{N}(k) \Vert_{2}^{2+\delta}\right] = 0, \quad \text{some} \,\, \delta > 0.
\]
Let
\[
Q = \lim_{N \to \infty} \mathbb{E}\left[Z_{N}Z_{N}^{T}\right].
\]
Then
\[
Z_{N} \xrightarrow{d}  \mathcal{N}\left(0, Q \right)
\]
\end{lem}

\begin{lem}\label{lemma_2} (Lemma 9.A.2 in \cite{ljung1987theory})
Let 
\[
S_{N} = Z_{M}(N) + X_{M}(N), \quad M,N = 1,2,\ldots
\]
such that
\begin{align*}
& \mathbb{E}\left[ X_{M}^{2}(N) \right] \leq C_{M}, \\
& \lim\limits_{M \to \infty} C_{M} = 0, \\
& \mathbb{P}\left( Z_{M}(N) \leq z \right) = F_{M,N}(z), \\
& \lim\limits_{N \to \infty} F_{M,N}(z) = F_{M}(z), \\
& \lim\limits_{M \to \infty} F_{M}(z) = F(z).
\end{align*}
Then,
\[
 \lim\limits_{N \to \infty}   \mathbb{P}\left(S_{N} \leq z \right) = F(z).
\]
\end{lem}

\begin{lem}\label{lemma_aaa111} (Eq. (7.91) in \cite{ljung1987theory}) Under Assumptions \ref{ass_a}-\ref{ass_c}, the CRLB is $\sigma_{\mathrm{CR}}(k) =  \delta_{d}^{2}(\sum_{l=1}^{k}\mathbb{E}[\psi_{l}(\theta)\psi_{l}(\theta)^{T}])^{-1}$, where $\psi_{l}(\theta) = \partial (y_{k} - \phi_{k}^{T}\vartheta)/ \partial \vartheta |_{\vartheta = \theta}$.
\end{lem}

\begin{lem}\label{lemma_c2233} (Eq. (11.19) in \cite{ljung1987theory})
The RLS estimate can be expressed as $\hat{\theta}_{k} = P_{k}\sum_{l=1}^{k}\phi_{l}y_{l} + P_{k}P_{0}^{-1}\hat{\theta}_{0}.$
\end{lem}

%Although we have established the asymptotic normality of the RLS algorithm, this does not guarantee that the $k$-estimation error variance achieves \(\delta_{d}^{2}(\bar{\mathbb{E}}[ \phi_{k}\phi_{k}^{T}])^{-1}\). 
%To address this challenge, we refer to Theorem 4.5.2 in \cite{chung2000course} (as shown as Lemma \ref{lemma_aaa222} in Appendix) and conduct a further investigation into the fourth-order convergence analysis of the RLS algorithm.

%\begin{cor}\label{lemma_a111}
%Under Assumptions \ref{ass_a}-\ref{ass_c}, for any positive constant $\varepsilon$,
%\begin{align}\label{ttt222}
%\left\Vert \hat{\theta}_{k} - \hat{\theta}_{k-1} \right\Vert_{2} = O\left(\frac{1}{k^{1-\varepsilon}}\right),  \quad \mathrm{a.s.} 
%\end{align}
%\end{cor}

\end{document}